\documentclass[11pt,a4paper]{article}
\usepackage{epsf,epsfig,amsfonts,amsgen,amsmath,amstext,amsbsy,amsopn,amsthm}
\usepackage{amsmath}
\usepackage{enumerate}
\usepackage{hhline}
\usepackage{multirow}
\usepackage{multicol}
\usepackage{booktabs}
\usepackage{lineno}
\usepackage{amsfonts,amsthm,amssymb,bm}
\usepackage{amsfonts}
\usepackage{graphics}
\usepackage{latexsym,bm}
\usepackage{amsfonts,amsthm,amssymb,bbding}
\usepackage{indentfirst}
\usepackage{graphicx}
\usepackage{color}
\usepackage[colorlinks=true,anchorcolor=blue,filecolor=blue,linkcolor=red,urlcolor=blue,citecolor=blue]{hyperref}
\usepackage{float}
\usepackage{tikz,enumerate}
\usetikzlibrary{calc}
\usepackage{geometry}
\newcommand{\n}{\noindent}

\newtheorem{thm}{Theorem}[section]

\newtheorem{lemma}[thm]{Lemma}

\newtheorem{problem}[thm]{Problem}

\newtheorem{proposition}[thm]{Proposition}

\renewcommand{\le}{\leqslant}
\renewcommand{\geq}{\geqslant}
\renewcommand{\ge}{\geqslant}

\begin{document}

\title{On a problem of Sivaraman and a problem of Gy\'arf\'as}

\author{
Kaiyang Lan\thanks{School of Mathematics and Statistics, Minnan Normal University, Zhangzhou, 363000, Fujian, China. Email: \url{kylan95@126.com}.} 
\and 
Wenlong Zhong\thanks{School of Mathematics and Statistics, Minnan Normal University, Zhangzhou, 363000, Fujian, China. Email: \url{2364810512@qq.com}.
} 
 }

\date{\today}
\maketitle



\begin{abstract}
	The \textit{girth} of a graph $G$, denoted $\mathrm{g}(G)$, is the length of a shortest cycle in $G$.
	If $G$ contains no cycle, we define $\mathrm{g}(G)=\infty$.
	Sivaraman (2020) asked for the optimal $\chi$-bounding function for the class of graphs whose complements have girth at least $6$.
	Let \(F(s) = \max\{\chi(G): \omega(G)\le s,\ \mathrm{g}(\overline{G})\ge 6\}\).
	We prove that there exists a constant \(c>0\) such that 
	\[
		c\left(\frac{s}{\log s}\right)^{4/3}
		\le
		F(s)
		\le
		(1+o(1))\frac{s^{3/2}}{\log s}.
	\]
	For small values, we establish the exact results
	\[
	F(1)=1,\; F(2)=2,\; F(3)=4,\; F(4)=5,\; F(5)=6,\; F(6)=8,
	\]
	and each bound is sharp.
	
	A graph $G$ is \emph{almost perfect} if every induced subgraph $H$ of $G$ satisfies
	\(\alpha(H)\omega(H)+1\ge |V(H)|\).
	Gy\'arf\'as (2023) asked whether almost perfect graphs are $\chi$-bounded by the function $g(x)=x+1$.
	We answer this question in the negative by showing that there is no constant $c$ such that every almost perfect
	graph $G$ satisfies $\chi(G)\le \omega(G)+c$.	
\end{abstract}

{\bf Keywords: girth, $\chi$-boundedness, almost perfect graph} 

{\bf 2020 AMS Subject Classifications: 05C15, 05C75} 

\section{Introduction}

All graphs in this note are finite, simple, and undirected.
The \emph{chromatic number} $\chi(G)$ of a graph $G$ is the minimum number of colors needed to color its vertices so that no two adjacent vertices share the same color.
A \emph{clique} (resp. an \emph{independent set}) in $G$ is a set of pairwise adjacent (resp. nonadjacent) vertices, and the \emph{clique number} $\omega(G)$ (resp. \emph{independence number} $\alpha(G)$) is the size of a largest clique (resp. independent set) in $G$.
It is clear that $\chi(G) \geq \omega(G)$, but in general $\chi(G)$ cannot be bounded from above by any function of $\omega(G)$; for instance, there are triangle-free graphs with arbitrarily large chromatic number~\cite{Erdos1959,Mycielski1955,Zykov1949}.
Empty induced subgraphs will not play any role; when they occur formally, we use the harmless
conventions $\alpha(\varnothing)=\omega(\varnothing)=\chi(\varnothing)=0$.

A graph class $\mathcal{F}$ is \emph{$\chi$-bounded}~\cite{Gyarfas1975} if there exists a function $f: \mathbb{N} \to \mathbb{N}$ such that for every $G \in \mathcal{F}$ and every induced subgraph $H$ of $G$, we have $\chi(H) \le f(\omega(H))$.
Such a function $f$ is called a \emph{$\chi$-bounding function} for $\mathcal{F}$.
The class $\mathcal{F}$ is \emph{polynomially $\chi$-bounded} if $f$ can be chosen to be bounded above by a polynomial function in $\omega(H)$; equivalently, there exists a polynomial function $p$ such that $\chi(H) \le p(\omega(H))$ for every $G \in \mathcal{F}$ and every induced subgraph $H$ of $G$.
Esperet~\cite{Esperet2017} conjectured that every $\chi$-bounded class is polynomially $\chi$-bounded. This conjecture was recently disproved by Bria\'nski, Davies, and Walczak~\cite{BrianskiDaviesWalczak2024}, who constructed $\chi$-bounded classes that are not polynomially $\chi$-bounded.

\subsection{The class of graphs whose complements have girth at least 6}

The \textit{girth} of a graph $G$, denoted $\mathrm{g}(G)$, is the length of a shortest cycle in $G$.
If $G$ contains no cycle, we define $\mathrm{g}(G)=\infty$.

The study of $\chi$-boundedness seeks to understand when the chromatic number of a graph can be controlled by its clique number. While the general question is hopeless (triangle-free graphs can have arbitrarily large chromatic number, see~\cite{Erdos1959,Mycielski1955,Zykov1949}), many natural graph classes have been shown to be $\chi$-bounded. One such class arises from a forbidden induced subgraph condition on the complement: graphs whose complements have large girth. Sivaraman \cite{Sivaraman2020} posed the following problem concerning the optimal $\chi$-bounding function for this class.

\begin{problem}[\cite{Sivaraman2020}]\label{problem11}
	What is the optimal $\chi$-bounding function for the class of graphs whose complements have girth at least $6$?
\end{problem}

For a fixed graph \(H\), we say that a graph \(G\) is \textit{\(H\)-free} if \(G\) has no induced subgraph isomorphic to \(H\).
More generally, for a family \(\mathcal{H}\) of graphs, \(G\) is \textit{\(\mathcal{H}\)-free} if it contains no induced subgraph isomorphic to any graph in \(\mathcal{H}\).
For a graph $G$ and a positive integer $k$, let $kG$ denote the disjoint union of $k$ copies of $G$.
Wagon~\cite{Wagon1980} proved that if $G$ is $2K_2$-free, then \(\chi(G) \le \binom{\omega(G)+1}{2}\).
For \(t\ge 3\), let \(C_t\) denote the cycle on \(t\) vertices.
We write $\overline G$ for the complement of \(G\).
Note that \(G\) is \(2K_2\)-free if and only if \(\overline G\) is \(C_4\)-free.
Since graphs whose complements have girth at least $6$ form a subclass of $2K_2$-free graphs, 
the class of graphs in Problem~\ref{problem11} has a quadratic $\chi$-bounding function.

The first main result of this paper provides a complete answer to Problem~\ref{problem11} for graphs with \(\omega(G)\le s\) for all \(1\le s\le 6\), and gives asymptotic bounds for the general case.
Our first main result is stated as follows:

\begin{thm}\label{thm:main}
	Let
	\[
	F(s)=\max\{\chi(G):\omega(G)\le s,\ \mathrm{g}(\overline G)\ge 6\}.
	\]
	Then:
	\begin{itemize}
		\item[(i)] There exists a constant \(c>0\) such that
		\[
		c\left(\frac{s}{\log s}\right)^{4/3}
		\le
		F(s)
		\le
		(1+o(1))\frac{s^{3/2}}{\log s}.
		\]
		\item[(ii)] The following exact values hold:
		\[
		F(1)=1,\; F(2)=2,\; F(3)=4,\; F(4)=5,\; F(5)=6,\; F(6)=8,
		\]
		and each bound is sharp.
	\end{itemize}
\end{thm}
	
\subsection{Almost perfect graphs}

A graph $G$ is \textit{perfect} if for every induced subgraph $H$ of $G$, we have $\chi(H)=\omega(H)$.
Perfect graphs are one of the most studied classes in graph theory, with deep structural results including the Strong Perfect Graph Theorem~\cite{chudnovsky2006}.
Lov\'asz~\cite{lovasz1972characterization} proved the following celebrated characterization of perfect graphs.

\begin{thm}[\cite{lovasz1972characterization}]\label{Lov\'asz}
	A graph $G$ is perfect if and only if every induced subgraph $H$ of $G$ satisfies
	\begin{equation}
		\alpha(H)\omega(H) \ge |V(H)|.
		\label{eq:perfect}
	\end{equation}
\end{thm}

The inequality~\eqref{eq:perfect} is necessary because in any graph $H$, a proper coloring partitions $V(H)$ into $\chi(H)$ independent sets, each of size at most $\alpha(H)$, so $|V(H)| \le \chi(H)\alpha(H)$.
For perfect graphs, $\chi(H)=\omega(H)$, yielding $|V(H)| \le \alpha(H)\omega(H)$.
The remarkable part of Theorem~\ref{Lov\'asz} is that this simple inequality is also sufficient for perfection.

A natural relaxation of perfect graphs, proposed by Gy\'arf\'as~\cite{Gyarfas2023}, is obtained by weakening the inequality in \eqref{eq:perfect} by one.
	A graph $G$ is \emph{almost perfect} if every induced subgraph $H$ of $G$ satisfies
	\begin{equation}
		\alpha(H)\omega(H) + 1 \ge |V(H)|.
		\label{eq:almostperfect}
	\end{equation}
In other words, almost perfect graphs are those whose induced subgraphs violate the perfectness inequality by at most one.
This definition captures graphs that are ``one step away'' from being perfect in the sense of Lov\'asz's characterization.

It is immediate that every perfect graph is almost perfect.
The converse is false; for example, odd cycles of length at least $5$ are almost perfect but not perfect.
Indeed, for $C_{2k+1}$ with $k\ge 2$, we have $\alpha(C_{2k+1})=k$, $\omega(C_{2k+1})=2$, and $|V(C_{2k+1})|=2k+1$, so $\alpha(C_{2k+1})\omega(C_{2k+1})+1=2k+1=|V(C_{2k+1})|$, while $\chi(C_{2k+1})=3>2=\omega(C_{2k+1})$.

Since almost perfect graphs are defined by a condition very close to that of perfect graphs, and perfect graphs are $\chi$-bounded by the identity function $f(x)=x$, it is natural to ask whether they inherit this strong property. The following question was posed by Gy\'arf\'as.

\begin{problem}[\cite{Gyarfas2023}]\label{p2}
	Are almost perfect graphs $\chi$-bounded by the function $g(x)=x+1$?
\end{problem}

Equivalently, the question asks whether every almost perfect graph $G$ satisfies $\chi(G) \le \omega(G)+1$.
Scott and Seymour \cite{scott2020induced} proved that almost perfect graphs are $\chi$-bounded (indeed, their result implies that they are $\theta$-bounded as well, see \cite{Gyarfas2023}), but their proof does not give the specific linear bound $x+1$.
The case $\omega(G)=2$ (or $\alpha(G)=2$) was confirmed by Gy\'arf\'as~\cite{Gyarfas2023} using Folkman's theorem: in that case, the almost perfect condition implies $\chi(G)\le 3=\omega(G)+1$.
Whether the $x+1$ bound holds for all almost perfect graphs remained open.

Our second result constructs, for every positive integer $m$, an almost perfect graph $G_m$ with $\chi(G_m)-\omega(G_m)=m$, thereby answering Problem~\ref{p2} in the negative.
In particular, the additive gap $\chi(G)-\omega(G)$ can be arbitrarily large among almost perfect graphs, so no constant $c$ can bound $\chi(G)$ by $\omega(G)+c$.

\begin{thm}\label{T2}
	For every positive integer $m$, there exists an almost perfect graph
	$G_m$ such that
	\[
	\chi(G_m)-\omega(G_m)=m.
	\]
	Consequently, there is no constant $c$ such that every almost perfect
	graph $G$ satisfies $\chi(G)\le \omega(G)+c$.
\end{thm}	
	

	The rest of the paper is organized as follows.
	In Section~\ref{sec:tools}, we establish the small Ramsey-type estimates needed for the proof.
	In Section~\ref{part1}, we prove the asymptotic bounds in Theorem~\ref{thm:main}(i).
	In Section~\ref{part2}, we complete the proof of Theorem~\ref{thm:main} by establishing the exact small values in part (ii).
	Finally, in Section~\ref{T2proof}, we prove Theorem~\ref{T2}, which gives a negative answer to Problem~\ref{p2}.

In the remainder of this section, we describe notation and terminology used in the paper.
For a graph $G$, write $\delta(G)$ for its minimum degree, and $\alpha(G)$ for its independence number.
For any $S \subseteq V(G)$, $G[S]$ denotes the subgraph of $G$ induced by $S$.
For any $v \in V(G)$, $N_G(v)$ denotes the neighborhood of $v$ and $d_G(v) = |N_G(v)|$ is the degree of $v$ in $G$.
Let \(N_G[v]:=N_G(v)\cup \{v\}\).
For any positive integer $i$, let $N_G^i(v) := \{u \in
V(G) \setminus \{v\} : d_G(u, v) = i\}$, where $d_G(u, v)$ is the distance between $u$ and $v$ in $G$.
Then $N_G^1(v) = N_G(v)$ is the neighborhood of $v$ in $G$.
When \(G\) is clear from the context, we ignore the subscript \(G\).

\section{Tools}\label{sec:tools}

A \textit{matching} in a graph \(G\) is a set of pairwise vertex-disjoint edges. 
The \textit{maximum matching number} \(\nu(G)\) is the maximum size (number of edges) of a matching in \(G\).
We will need the following relation between the matching number of a triangle-free graph and the chromatic number of its complement.

\begin{lemma}\label{lem:coloring-matching}
	If $H$ is triangle-free, then \(\chi(\overline H)=|V(H)|-\nu(H)\).
\end{lemma}

\begin{proof}
	A proper coloring of $\overline H$ is a partition of $V(H)$ into stable sets of $\overline H$.  These are exactly cliques of $H$.  Since $H$ is triangle-free, every clique of $H$ has size at most $2$.  Hence every color class of $\overline H$ is either a single vertex of $H$ or an edge of $H$.
	Therefore, a coloring of $\overline H$ with $t$ colors is the same as a partition of $V(H)$ into $t$ parts, each part being either one vertex or two adjacent vertices of $H$.  If $m$ of the parts have size $2$, then those $m$ parts are pairwise disjoint edges of $H$, hence a matching, and
	\[
	t=m+(|V(H)|-2m)=|V(H)|-m.
	\]
	Minimizing $t$ is therefore the same as maximizing $m$.
	Thus \(\chi(\overline H)=|V(H)|-\nu(H)\).
\end{proof}

\medskip

It is convenient to pass to the complement.  Put
\[
H=\overline G.
\]
Then
\[
\mathrm{g}(\overline G)\ge 6\quad\Longleftrightarrow\quad H\text{ is $\{C_3,C_4,C_5\}$-free},
\]
and
\[
\omega(G)=\alpha(H).
\]

For $s\ge 1$, we further define the \textit{girth-six Ramsey extremal number}
\[
R_6(s)=\max\{|V(H)|: \mathrm{g}(H)\ge 6,
\ \alpha(H)\le s\}.
\]
Equivalently,
\[
R_6(s)=R(\{C_3,C_4,C_5\},K_{s+1})-1.
\]

The exact small clique cases below require several small estimates for $R_6(s)$.

\begin{lemma}\label{lem:Rsmall}
	The following bounds hold:
	\[
	\begin{aligned}
		R_6(1)&=2, & R_6(2)&=4, & R_6(3)&=7, & R_6(4)&=9, \\
		R_6(5)&\le 12, & R_6(6)&\le 16.
	\end{aligned}
	\]
	Moreover, the equality cases $R_6(3)=7$ and $R_6(4)=9$ are realized by $C_7$ and $C_9$, respectively.
\end{lemma}

\begin{proof}
	We prove each estimate carefully.
	
	First, $R_6(1)=2$.  If $\alpha(H)\le1$, then every two vertices are adjacent.  Since $H$ is triangle-free, this gives $|V(H)|\le2$, and $K_2$ shows equality.
	
	Next, $R_6(2)=4$.
	Let $H$ have girth at least six and $\alpha(H)\le2$.
	Choose any vertex $v$.
	Since $H$ is triangle-free, $N_H(v)$ is a stable set, so \(d_H(v) \le2\).
	Let \(A=V(H)\setminus N_H[v]\).
	If $A$ contained two nonadjacent vertices $x,y$, then $\{v,x,y\}$ would be a stable set of size three, contradicting $\alpha(H)\le2$.
	Hence $A$ is a clique.
	Since $H$ is triangle-free, $|A|\le2$.  Thus
	\[
	|V(H)|=1+d_H(v)+|A|\le1+2+2=5.
	\]
	If equality held, then every vertex would have degree two, and the graph would be a \(C_5\).
	This is impossible because the girth is at least six.
	Hence $|V(H)|\le4$.
	The graph $2K_2$ has four vertices, girth infinity, and independence number two, so $R_6(2)=4$.
	
	Now prove $R_6(3)=7$.
	Let $H$ have girth at least six and $\alpha(H)\le3$.
	For any vertex $v$, $N_H(v)$ is stable, so $d_H(v)\le3$.
	Let \(A=V(H)\setminus N_H[v]\).
	We have $\alpha(H[A])\le2$, because a stable set of size three in $A$ together with $v$ would give a stable set of size four in $H$.
	By the already proved $R_6(2)=4$, we get $|A|\le4$.  Hence
	\[
	|V(H)|\le1+3+4=8.
	\]
	If $|V(H)|=8$, then no vertex can have degree at most two; otherwise the same argument gives
	\[
	|V(H)|\le1+2+4=7.
	\]
	Thus $\delta(H)\ge3$.
	Since $N(v)$ is stable and $\alpha(H)\le3$, also $d_H(v)\le3$ for every $v$, and therefore $H$ is 3-regular.
	But in a 3-regular graph of girth at least six, the set \(N_H^2(v)\) (for any vertex \(v\in V(H)\)) has at least
	\[
	1+3+3\cdot2=10
	\]
	distinct vertices: the vertices in \(N_H(v)\) are distinct and nonadjacent, and no two vertices in \(N_H(v)\) have a common neighbor in \(N_H^2(v)\), otherwise there is a $C_4$.
	This contradicts $|V(H)|=8$.
	Hence $|V(H)|\le7$.
	The graph $C_7$ has girth seven and independence number three, so $R_6(3)=7$.
	
	It will be useful to note the equality case for $R_6(3)$.
	If $|V(H)|=7$ and $\alpha(H)\le3$, then $H$ is not a forest, because every forest is bipartite and one side of a bipartition of seven vertices has size at least four.
	Hence $H$ has a cycle.
	The girth is at least six, so the cycle has length six or seven.
	If there is a $C_7$, then it uses all vertices and has no chord, so $H=C_7$.
	If there is a $C_6$, let $x$ be the seventh vertex of \(H\).
	The vertex $x$ has at most one neighbor on that $C_6$; otherwise two neighbors on the $C_6$ have distance one, two, or three along the cycle, producing a $C_3$, $C_4$, or $C_5$.
	Then one can choose three pairwise nonadjacent vertices of the $C_6$ that are not adjacent to $x$, and together with $x$ they form a stable set of size four, a contradiction.
	Thus the equality case is $C_7$.
	
	Next, we prove $R_6(4)=9$.
	Let $H$ have girth at least six and $\alpha(H)\le4$.
	If $\delta(H)\ge3$, then for any vertex $v$ the set $N_H^2(v)$ has size at least \(3\cdot2=6\).
	Indeed, every neighbor of $v$ has at least two other neighbors, and different neighbors of $v$ cannot share such a neighbor in $N_H^2(v)$ because that would make a $C_4$.
	Also $N_H^2(v)$ is stable, since an edge inside $N_H^2(v)$ would create a $C_5$ or a shorter cycle.
	Thus $\alpha(H)\ge6$, a contradiction.
	Therefore $H$ has a vertex $v$ with $d(v)\le2$.
	
	Let $A=V(H)\setminus N_H[v]$.  Then $\alpha(H[A])\le3$, because a stable set of size four in $A$ together with $v$ would give a stable set of size five in $H$.
	Hence $|A|\le R_6(3)=7$, and so
	\[
	|V(H)|\le1+2+7=10.
	\]
	We now exclude $|V(H)|=10$.  In that case $d_H(v)=2$ and $H[A]$ has seven vertices with independence number at most three, hence $H[A]=C_7$ by the equality case above.  Write $N_H(v)=\{a,b\}$.  Since $H$ is triangle-free, $ab\notin E(H)$.  Each of $a$ and $b$ has at most one neighbor on the $C_7$, for two neighbors would again produce a $C_3$, $C_4$, or $C_5$ using a shortest arc on the $C_7$.  Thus at most two vertices of the $C_7$ are adjacent to $a$ or $b$.  After deleting those at most two vertices from $C_7$, the remaining graph contains a stable set of size three.  Those three vertices together with $a$ and $b$ form a stable set of size five, contradicting $\alpha(H)\le4$.  Hence $|V(H)|\le9$.  The cycle $C_9$ has girth nine and independence number four, so $R_6(4)=9$.
	
	Now let $\alpha(H)\le5$.
	If $\delta(H)\ge3$, the preceding argument for $N_H^2(v)$ gives a stable set of size at least six, which is impossible.
	Thus some vertex $v$ has degree at most two.
	For $A=V(H)\setminus N_H[v]$, we have $\alpha(H[A])\le4$, so $|A|\le R_6(4)=9$.
	Therefore
	\[
	|V(H)|\le1+2+9=12.
	\]
	This proves $R_6(5)\le12$.
	
	Finally let $\alpha(H)\le6$.
	If $\delta(H)\ge4$, then for any vertex $v$ the second neighborhood $N_H^2(v)$ has size at least \(4\cdot3=12\), and is stable, a contradiction.
	Thus some vertex $v$ has degree at most three.
	For $A=V(H)\setminus N_H[v]$, we have $\alpha(H[A])\le5$, so $|A|\le R_6(5)\le12$.
	Therefore
	\[
	|V(H)|\le1+3+12=16.
	\]
	This proves $R_6(6)\le16$.
\end{proof}

The exact values $F(4),F(5),F(6)$ require not only upper bounds on $|V(H)|$, but also lower bounds on $\nu(H)$.
We now record the finite calculation used later.

For $0\le m\le16$, define $\lambda(m)$ by
\[
\lambda(0)=0,
\]
and
\[
\begin{array}{c|cccccc}
	m & 1,2 & 3,4 & 5,6,7 & 8,9 & 10,11,12 & 13,14,15,16 \\
	\hline
	\lambda(m) & 1 & 2 & 3 & 4 & 5 & 6.
\end{array}
\]
By Lemma~\ref{lem:Rsmall}, we can deduce that every graph of girth at least six and order $m\le16$ has independence number at least $\lambda(m)$.

We use the Tutte--Berge formula~(see \cite{Berge1958,Tutte1947}) in the form
\begin{equation}
	|V(H)| - 2\nu(H) = \max_{S \subseteq V(H)} \bigl( o(H-S) - |S| \bigr),
	\label{eq:tutte-berge}
\end{equation}
where $o(H-S)$ is the number of odd components of $H-S$.

Suppose that $|V(H)|=n$ and that we want to prove $n-\nu(H)\le B$ for some constant \(B\) (in our applications $B$ will be $5$, $6$, or $8$).
If this fails, then \(n-\nu(H)\ge B+1\), so \(\nu(H)\le n-B-1\).
Tutte--Berge~\eqref{eq:tutte-berge} gives a set $S\subseteq V(H)$ such that
\[
o(H-S)-|S|\ge n-2(n-B-1)=2B+2-n.
\]
Let $t=|S|$.  If the component orders of $H-S$ are $m_1,\ldots,m_q$, then
\[
\sum_{i=1}^q m_i=n-t,
\]
and at least
\[
t+2B+2-n
\]
of the $m_i$ are odd.
Since different components have no edges between them, the union of stable sets from the components is stable in $H$.
Therefore
\[
\alpha(H)\ge\sum_{i=1}^q\lambda(m_i).
\]
The following table gives the minimum value of this lower bound over all choices of $t$ and all such component-order partitions.
The displayed partition is a partition of $n-t$ attaining the minimum.

\begin{table}[htbp]
	\centering
	\begin{tabular}{c c c c}
		\toprule
		Target $B$ & $n$ & minimizing component orders & forced lower bound for $\alpha(H)$ \\
		\midrule
		$5$ & $6$ & $1+1+1+1+1+1$ & $6$ \\
		$5$ & $7$ & $3+1+1+1+1$ & $6$ \\
		$5$ & $8$ & $5+1+1+1$ & $6$ \\
		$5$ & $9$ & $7+1+1$ & $5$ \\
		\midrule
		$6$ & $7$ & $1+1+1+1+1+1+1$ & $7$ \\
		$6$ & $8$ & $3+1+1+1+1+1$ & $7$ \\
		$6$ & $9$ & $5+1+1+1+1$ & $7$ \\
		$6$ & $10$ & $7+1+1+1$ & $6$ \\
		$6$ & $11$ & $9+1+1$ & $6$ \\
		$6$ & $12$ & $11+1$ & $6$ \\
		\midrule
		$8$ & $9$ & $1+1+1+1+1+1+1+1+1$ & $9$ \\
		$8$ & $10$ & $3+1+1+1+1+1+1+1$ & $9$ \\
		$8$ & $11$ & $5+1+1+1+1+1+1$ & $9$ \\
		$8$ & $12$ & $7+1+1+1+1+1$ & $8$ \\
		$8$ & $13$ & $9+1+1+1+1$ & $8$ \\
		$8$ & $14$ & $11+1+1+1$ & $8$ \\
		$8$ & $15$ & $7+7+1$ & $7$ \\
		$8$ & $16$ & $15+1$ & $7$ \\
		\bottomrule
	\end{tabular}
	\caption{Tutte--Berge~\eqref{eq:tutte-berge} calculations for $B=5,6,8$.}
	\label{tab:tb-calc}
\end{table}

For example, in the row $B=6,n=12$, the failure of $n-\nu(H)\le6$ gives a set $S$ for which $H-S$ has at least $|S|+2$ odd components.  The minimum possible contribution to the independence number, over all choices of $S$ and all component orders, is attained by component orders $11+1$, giving $\lambda(11)+\lambda(1)=5+1=6$.
The other rows in~\autoref{tab:tb-calc} are calculated in the same manner from the displayed values of $\lambda$.

\section{The asymptotic bounds}\label{part1}

In this section, we prove part (i) of Theorem~\ref{thm:main}. 
We first derive the general upper bound for \(F(s)\) via the complement reduction and the second neighborhood argument, then obtain the lower bound from Spencer's Ramsey construction (see \cite{Spencer1977}).

\begin{lemma}\label{lem:Exact reduction to R_6}
	For every $s\ge 1$,
	\[
	\frac12R_6(s)\le F(s)\le \frac12\bigl(R_6(s)+s\bigr).
	\]
	Consequently, whenever $R_6(s)/s\to\infty$,
	\[
	F(s)=\left(\frac12+o(1)\right)R_6(s).
	\]
\end{lemma}

\begin{proof}
	For the lower bound, take a graph $H$ with $\mathrm{g}(H)\ge6$, $\alpha(H)\le s$, and $|V(H)|=R_6(s)$.
	Since \(\nu(H)\le |V(H)|/2\), Lemma~\ref{lem:coloring-matching} gives
	\[
	\chi(\overline H)=|V(H)|-\nu(H)\ge |V(H)|/2=R_6(s)/2.
	\]
	Since $\omega(\overline H)=\alpha(H)\,\le s$, this gives $F(s)\ge R_6(s)/2$.
	
	For the upper bound, let $H$ be any graph with $\mathrm{g}(H)\ge6$ and $\alpha(H)\le s$.  Let $M$ be a maximal matching in $H$.
	The vertices not covered by $M$ form a stable set: if two uncovered vertices were adjacent, the edge between them could be added to the matching.
	Hence at most $s$ vertices are uncovered.
	Therefore \(2\nu(H)\ge |V(H)|-s\),
	so
	\[
	\nu(H)\ge \frac{|V(H)|-s}{2}.
	\]
	By Lemma~\ref{lem:coloring-matching},
	\[
	\chi(\overline H)=|V(H)|-\nu(H)
	\le |V(H)|-\frac{|V(H)|-s}{2}
	=\frac{|V(H)|+s}{2}
	\le \frac{R_6(s)+s}{2}.
	\]
	This proves the claimed upper bound.
\end{proof}

Recall that
\[
F(s)=\max\{\chi(G):\omega(G)\le s,\ \mathrm{g}(\overline G)\ge 6\}.
\]

\begin{proposition}\label{thm:asymptotic bounds}
	There exists a constant \(c>0\) such that for all \(s\ge 2\),
	\[
	c\left(\frac{s}{\log s}\right)^{4/3}
	\le
	F(s)
	\le
	(1+o(1))\frac{s^{3/2}}{\log s}.
	\]
\end{proposition}
\begin{proof}
	The lower bound is supplied by Spencer's short-cycle Ramsey construction (see \cite{Spencer1977}).
	It gives a constant $c>0$ such that
	\[
	R_6(s)
	\ge
	c\left(\frac{s}{\log s}\right)^{4/3}.
	\]
	Therefore, Lemma~\ref{lem:Exact reduction to R_6} gives
	\[
	F(s) \ge \frac{1}{2}R_6(s) \ge c'\left(\frac{s}{\log s}\right)^{4/3},
	\]
	where \(c' = \frac{c}{2} > 0\).
	
	For the upper bound, let \(H\) be a graph with \(\mathrm{g}(H)\ge 6\), \(\alpha(H)\le s\), \(|V(H)|=n\), and average degree \(d\). 
	For each vertex \(v\), the second neighborhood \(N_H^2(v)\) is stable; indeed, if two vertices of \(N_H^2(v)\) were adjacent, then together with their parents in \(N_H(v)\) and the vertex \(v\) they would form a cycle of length at most \(5\), contradicting \(\mathrm{g}(H)\ge 6\).
	Moreover, every vertex of \(N_H^2(v)\) has a unique neighbor in \(N_H(v)\); otherwise, two vertices of \(N_H(v)\) would share a common neighbor in \(N_H^2(v)\), creating a \(C_4\).
	Thus
	\[
	|N_H^2(v)| = \sum_{u\in N_H(v)}(d_H(u)-1) \le s.
	\]
	Summing over all vertices gives
	\[
	\sum_{v\in V(H)} |N_H^2(v)| = \sum_{u\in V(H)} d_H(u)(d_H(u)-1) \le ns.
	\]
	Hence
	\[
	\frac{1}{n}\sum_{u\in V(H)} d_H(u)^2 - d \le s.
	\]
	By Cauchy's inequality 
	\[
	\left(\frac{1}{n}\sum_{u\in V(H)} d_H(u)\right)^2
	\le \frac{1}{n}\sum_{u\in V(H)} d_H(u)^2\], 
	we have
	\[
	d^2 \le \frac{1}{n}\sum_{u\in V(H)} d_H(u)^2.
	\]
	So
	\[
	d^2 - d \le s,
	\]
	and solving the quadratic inequality gives
	\[
	d \le \sqrt{s} + 1.
	\]
	Since \(H\) is triangle-free, the Ajtai--Koml\'os--Szemer\'edi--Shearer independent set estimate (see \cite{Ajtai1980,Shearer1983}) gives
	\[
	\alpha(H) \ge (1+o(1))\frac{n\log d}{d}
	\]
	as \(d\to\infty\). Since \(\alpha(H)\le s\), it follows that
	\[
	s \ge (1+o(1))\frac{n\log d}{d}.
	\]
	Using \(d \le (1+o(1))\sqrt{s}\), we have \(\log d = \left(\frac{1}{2}+o(1)\right)\log s\). 
	From \(s \ge (1+o(1))\frac{n\log d}{d}\), it follows that
	\[
	n \le (2+o(1))\frac{s^{3/2}}{\log s}.
	\]
	Thus
	\[
	R_6(s) \le (2+o(1))\frac{s^{3/2}}{\log s}.
	\]
	By Lemma~\ref{lem:Exact reduction to R_6},
	\[
	F(s) \le \frac{1}{2}(R_6(s) + s) \le (1+o(1))\frac{s^{3/2}}{\log s}.
	\]
	
	Combining the lower and upper bounds gives
	\[
	c\left(\frac{s}{\log s}\right)^{4/3}
	\le
	F(s)
	\le
	(1+o(1))\frac{s^{3/2}}{\log s}.
	\]
	This proves Proposition~\ref{thm:asymptotic bounds}.
\end{proof}

\section{Exact small values}\label{part2}

In this section, we prove the second part of Theorem~\ref{thm:main}.
The proof relies on the Tutte--Berge formula~\eqref{eq:tutte-berge} and the finite calculation recorded in Section~\ref{sec:tools}.

\begin{proposition}\label{thm:F4}
	If $G$ satisfies $\mathrm{g}(\overline G)\ge6$ and is $K_5$-free, then \(\chi(G)\le5\).
	The bound is sharp.  Equivalently, \(F(4)=5\).
\end{proposition}

\begin{proof}
	Let $H=\overline G$.  Then $\mathrm{g}(H)\ge6$ and
	\[
	\alpha(H)=\omega(G)\le4.
	\]
	By Lemma~\ref{lem:Rsmall},
	\[
	n:=|V(H)|\le R_6(4)=9.
	\]
	Since $H$ is triangle-free, Lemma~\ref{lem:coloring-matching} gives
	\[
	\chi(G)=\chi(\overline H)=n-\nu(H).
	\]
	It remains to prove $n-\nu(H)\le5$.
	
	If $n\le5$, this is immediate.  Suppose $6\le n\le9$ and assume, for a contradiction, that
	\[
	n-\nu(H)\ge6.
	\]
	Then the finite Tutte--Berge~\eqref{eq:tutte-berge} calculation from~\autoref{tab:tb-calc} with target $B=5$ gives
	\[
	\alpha(H)\ge
	\begin{cases}
		6, & n=6,7,8,\\
		5, & n=9.
	\end{cases}
	\]
	In every case $\alpha(H)\ge5$, contradicting $\alpha(H)\le4$.  Hence $n-\nu(H)\le5$, and therefore \(\chi(G)\le5\) by Lemma~\ref{lem:coloring-matching}.

	For sharpness, take
	\[
	H=C_9,
	\qquad
	G=\overline{C_9}.
	\]
	Then $\mathrm{g}(\overline G)=\mathrm{g}(H)=9\ge6$ and
	\[
	\omega(G)=\alpha(H)=\alpha(C_9)=4,
	\]
	so $G$ is $K_5$-free.  Also
	\[
	\nu(H)=\nu(C_9)=4,
	\]
	and hence
	\[
	\chi(G)=\chi(\overline H)=|V(H)|-\nu(H)=9-4=5.
	\]
	by Lemma~\ref{lem:coloring-matching}.
	Thus the bound is attained.
\end{proof}

\begin{proposition}\label{thm:F5}
	If $G$ satisfies $\mathrm{g}(\overline G)\ge6$ and is $K_6$-free, then \(\chi(G)\le6\).
	The bound is sharp.
	Equivalently, \(F(5)=6\).
\end{proposition}

\begin{proof}
	Let $H=\overline G$.  Then $\mathrm{g}(H)\ge6$ and
	\[
	\alpha(H)=\omega(G)\le5.
	\]
	By Lemma~\ref{lem:Rsmall},
	\[
	n:=|V(H)|\le R_6(5)\le12.
	\]
	Again, Lemma~\ref{lem:coloring-matching} gives
	\[
	\chi(G)=\chi(\overline H)=n-\nu(H).
	\]
	We prove $n-\nu(H)\le6$.
	
	If $n\le6$, this is immediate.  Suppose $7\le n\le12$ and assume, for a contradiction, that
	\[
	n-\nu(H)\ge7.
	\]
	Then the finite Tutte--Berge~\eqref{eq:tutte-berge} calculation from~\autoref{tab:tb-calc} with target $B=6$ gives
	\[
	\alpha(H)\ge
	\begin{cases}
		7, & n=7,8,9,\\
		6, & n=10,11,12.
	\end{cases}
	\]
	In every case $\alpha(H)\ge6$, contradicting $\alpha(H)\le5$.
	Therefore $n-\nu(H)\le6$, and so \(\chi(G)\le6\) by Lemma~\ref{lem:coloring-matching}.
	
	For sharpness, take
	\[
	H=C_{11},
	\qquad
	G=\overline{C_{11}}.
	\]
	Then
	\[
	\mathrm{g}(\overline G)=\mathrm{g}(H)=11\ge6,
	\]
	and
	\[
	\omega(G)=\alpha(H)=\alpha(C_{11})=5,
	\]
	so $G$ is $K_6$-free.  Furthermore,
	\[
	\nu(H)=\nu(C_{11})=5.
	\]
	Thus
	\[
	\chi(G)=\chi(\overline H)=|V(H)|-\nu(H)=11-5=6.
	\]
	by Lemma~\ref{lem:coloring-matching}.
	This proves sharpness.
\end{proof}

\begin{proposition}\label{thm:F6}
	If $G$ satisfies $\mathrm{g}(\overline G)\ge6$ and is $K_7$-free, then \(\chi(G)\le8\).
	The bound is sharp.
	Equivalently, \(F(6)=8\).
\end{proposition}

\begin{proof}
	Let $H=\overline G$.  Then $\mathrm{g}(H)\ge6$ and
	\[
	\alpha(H)=\omega(G)\le6.
	\]
	By Lemma~\ref{lem:Rsmall},
	\[
	n:=|V(H)|\le R_6(6)\le16.
	\]
	As before, Lemma~\ref{lem:coloring-matching} gives
	\[
	\chi(G)=\chi(\overline H)=n-\nu(H).
	\]
	We prove $n-\nu(H)\le8$.
	
	If $n\le8$, this is immediate.
	Suppose $9\le n\le16$ and assume, for a contradiction, that
	\[
	n-\nu(H)\ge9.
	\]
	Then the finite Tutte--Berge~\eqref{eq:tutte-berge} calculation from~\autoref{tab:tb-calc} with target $B=8$ gives
	\[
	\alpha(H)\ge
	\begin{cases}
		9, & n=9,10,11,\\
		8, & n=12,13,14,\\
		7, & n=15,16.
	\end{cases}
	\]
	In every case $\alpha(H)\ge7$, contradicting $\alpha(H)\le6$.
	Therefore $n-\nu(H)\le8$, and consequently Lemma~\ref{lem:coloring-matching} gives
	\[
	\chi(G)\le8.
	\]
	
	For sharpness, take
	\[
	H=2C_7,
	\quad
	G=\overline H.
	\]
	Then $\mathrm{g}(H)=7\ge6$, and
	\[
	\omega(G)=\alpha(H)=\alpha(C_7)+\alpha(C_7)=3+3=6.
	\]
	Thus $G$ is $K_7$-free.  Moreover,
	\[
	\nu(H)=\nu(C_7)+\nu(C_7)=3+3=6,
	\]
	and
	\[
	|V(H)|=14.
	\]
	Hence by Lemma~\ref{lem:coloring-matching}
	\[
	\chi(G)=\chi(\overline H)=|V(H)|-\nu(H)=14-6=8.
	\]
	This proves sharpness.
\end{proof}

Finally we deduce Theorem~\ref{thm:main}. 


\begin{proof}[\bf Proof of Theorem~\ref{thm:main}.]
	The asymptotic bounds follow from Proposition~\ref{thm:asymptotic bounds}.
	The values \(F(4),F(5),F(6)\) are established in Propositions~\ref{thm:F4}--\ref{thm:F6}.
	So it remains to show that \(F(1)=1,  F(2)=2,  F(3)=4\).
	
	The case \(F(1)=1\) is trivial.
	
	We first show that \(F(2)=2\).
	Suppose \(G\) satisfies \(\mathrm{g}(\overline G)\ge6\) and is triangle-free.
	Let \(H=\overline G\).
	Then \(\alpha(H)=\omega(G)\le 2\).
	By Lemma~\ref{lem:Rsmall},
	\[
	n:=|V(H)|\le R_6(2)=4.
	\]
	Since \(G\) is triangle-free and has at most four vertices, \(G\) is bipartite; hence \(\chi(G)\le 2\).
	For sharpness, take \(G=C_4\). Its complement is \(2K_2\), which has girth infinity, and \(\chi(G)=\chi(C_4)=2\).
	Thus the bound is attained, so \(F(2)=2\).
	
	Finally, we show that \(F(3)=4\).
	Suppose $G$ satisfies $\mathrm{g}(\overline G)\ge6$ and is \(K_4\)-free.
	Since \(G\) is \(2K_2\)-free if and only if \(\overline G\) is \(C_4\)-free, graphs whose complements have girth at least $6$ form a subclass of $2K_2$-free graphs.
	It follows that \(G\) is also \(2K_2\)-free.
	By a theorem of Gaspers and Huang~\cite{GaspersHuang2019}, every $\{2K_2,K_4\}$-free graph satisfies \(\chi(G)\le 4\).
	For sharpness, take
	\[
	H=C_7,
	\quad
	G=\overline{C_7}.
	\]
	Then $\mathrm{g}(\overline G)=\mathrm{g}(H)=7 \ge6$ and
	\[
	\omega(G)=\alpha(H)=\alpha(C_7)=3,
	\]
	so $G$ is $K_4$-free and
	\[
	\chi(G)=\chi(\overline{C_7})=4.
	\]
	Thus the bound is attained.
	Therefore, \(F(3)=4\).
	
	In conclusion, we complete the proof of Theorem~\ref{thm:main}.
\end{proof}

\section{Almost perfect graphs with $\chi-\omega$ arbitrarily large}\label{T2proof}

In this section, we prove Theorem~\ref{T2}. The construction uses the join operation, which we now define.
For \(m\ge 2\), the \emph{join} of pairwise vertex-disjoint graphs $G_1,\dots,G_m$, denoted
\[
G_1+G_2+\cdots+G_m,
\]
is the graph obtained from their disjoint union by adding all possible edges between vertices in distinct parts.

\begin{lemma}\label{lem:join-parameters}
	For any graphs $G_1,\dots,G_m$,
	\[
	\alpha(G_1+\cdots+G_m)=\max \{\alpha(G_1),\alpha(G_2),\ldots,\alpha(G_m)\},
	\]
	\[
	\omega(G_1+\cdots+G_m)=\sum_{i=1}^m \omega(G_i),
	\]
	and
	\[
	\chi(G_1+\cdots+G_m)=\sum_{i=1}^m \chi(G_i).
	\]
	Moreover, every induced subgraph of $G_1+\cdots+G_m$ is of the form
	$A_1+\cdots+A_m$, where each $A_i$ is an induced subgraph of $G_i$.
\end{lemma}

\begin{proof}
	An independent set in a join cannot use vertices from two distinct parts, because all cross-edges are present.
	Hence the largest independent set is contained in one part, proving the formula for $\alpha$.
	A clique in a join is obtained by taking a clique in each part and taking all of their vertices together, since all cross-edges are present.
	This proves the formula for clique number.
	Finally, the colors used on different parts must be disjoint, because all vertices in one part are adjacent to all vertices in every other part.
	Therefore the chromatic number is the sum of the chromatic numbers of the parts.
	The final assertion follows directly from the definition of the join.
\end{proof}

We shall also need the following elementary facts about induced subgraphs of odd cycles.

\begin{lemma}\label{lem:odd-cycle-induced}
	Let $k\ge 2$, and let $A$ be an induced subgraph of the odd cycle
	$C_{2k+1}$. If $A=C_{2k+1}$, then
	\[
	|V(A)|=2k+1,\qquad \alpha(A)=k,\qquad \omega(A)=2.
	\]
	If $A$ is a proper induced subgraph of $C_{2k+1}$, then
	\[
	|V(A)|\le \alpha(A)\omega(A).
	\]
	In every case, $\alpha(A)\le k$.
\end{lemma}

\begin{proof}
	The first assertion is the standard calculation for an odd cycle of length at least five.	
	Assume now that $A$ is a proper induced subgraph.
	Then $A$ is a forest, indeed a disjoint union of paths and isolated vertices.
	If $A$ has no edge, then $\omega(A)=1$ and $\alpha(A)=|V(A)|$, so $|V(A)|=\alpha(A)\omega(A)$.
	If $A$ has an edge, then $\omega(A)=2$.
	Since every forest is bipartite, one side of a bipartition has size at least $|V(A)|/2$, and hence $\alpha(A)\ge |V(A)|/2$.
	Therefore $\alpha(A)\omega(A)\ge |V(A)|$.
	Finally, every independent set of $A$ is also an independent set of $C_{2k+1}$, whose independence number is $k$.
	Hence $\alpha(A)\le k$.
\end{proof}

We now construct the graphs that will serve as our counterexamples.

\begin{proof}[\bf Proof of Theorem~\ref{T2}.]
	For \(m=1\), the statement is trivial (take \(G_1=C_5\), for which \(\chi(G_1)-\omega(G_1)=3-2=1\)).
	So we assume \(m\ge2\).
	For $1\le i\le m$, put $k_i=i+1$, and define
	\[
	G_m=C_{2k_1+1}+C_{2k_2+1}+\cdots+C_{2k_m+1}.
	\]
	Equivalently,
	\[
	G_m=C_5+C_7+\cdots+C_{2m+3}.
	\]
	By Lemma~\ref{lem:join-parameters}, and since each odd cycle $C_{2k_i+1}$ has chromatic number $3$ and clique number $2$, we have
	\[
	\chi(G_m)=3m,
	\qquad
	\omega(G_m)=2m.
	\]
	Thus $\chi(G_m)-\omega(G_m)=m$.
	
	It remains to prove that $G_m$ is almost perfect.
	Let $H$ be an arbitrary
	induced subgraph of $G_m$. By Lemma~\ref{lem:join-parameters}, there are
	induced subgraphs $A_i\subseteq C_{2k_i+1}$ such that
	\[
	H=A_1+\cdots+A_m.
	\]
	Write
	\[
	a_i=\alpha(A_i),\qquad b_i=\omega(A_i),\qquad n_i=|V(A_i)|,
	\]
	and set
	\[
	M=\max \{a_1,a_2,\ldots,a_m\}.
	\]
	Then, again by Lemma~\ref{lem:join-parameters},
	\[
	\alpha(H)=M,
	\qquad
	\omega(H)=\sum_{i=1}^m b_i,
	\qquad
	|V(H)|=\sum_{i=1}^m n_i.
	\]
	
	Let
	\[
	S=\{i: A_i=C_{2k_i+1}\}
	\]
	be the set of indices for which the whole odd cycle is present. By
	Lemma~\ref{lem:odd-cycle-induced}, for $i\notin S$ we have
	$n_i\le a_i b_i$, while for $i\in S$ we have
	$n_i=a_i b_i+1$. Hence
	\begin{equation}
		|V(H)|=\sum_{i=1}^m n_i
		\le \sum_{i=1}^m a_i b_i+|S|.
		\label{eq3.1}
	\end{equation}
	On the other hand,
	\[
	\begin{aligned}
		\alpha(H)\omega(H)+1
		&=M\sum_{i=1}^m b_i+1 \\
		&=\sum_{i=1}^m a_i b_i+
		\sum_{i=1}^m (M-a_i)b_i+1.
	\end{aligned}
	\]
	Therefore, by \eqref{eq3.1}, it suffices to prove the following inequality
	\begin{equation}
		\sum_{i=1}^m (M-a_i)b_i+1\ge |S|.
		\label{eq3.2}
	\end{equation}
	All terms in the sum are nonnegative. For $i\in S$, Lemma~\ref{lem:odd-cycle-induced}
	gives $a_i=k_i$ and $b_i=2$. Thus
	\begin{equation}
		\sum_{i=1}^m (M-a_i)b_i+1
		\ge 2\sum_{i\in S}(M-k_i)+1.
		\label{eq3.3}
	\end{equation}
	Let $r=|S|$. If $r=0$, the desired inequality~\eqref{eq3.2} is immediate. Suppose
	$r\ge 1$, and put
	\[
	K=\max_{i\in S} k_i.
	\]
	Since $M\ge K$ and the integers $k_i$ are pairwise distinct, the numbers
	$K-k_i$ for $i\in S$ are $r$ distinct nonnegative integers, one of which
	is $0$. Hence
	\[
	\sum_{i\in S}(M-k_i)
	\ge \sum_{i\in S}(K-k_i)
	\ge 0+1+\cdots+(r-1).
	\]
	Consequently,
	\[
	2\sum_{i\in S}(M-k_i)+1
	\ge 2(0+1+\cdots+(r-1))+1
	=r(r-1)+1
	\ge r=|S|.
	\]
	Thus, it follows from~\eqref{eq3.3} that
	\[
	\alpha(H)\omega(H)+1\ge |V(H)|.
	\]
	Since $H$ was arbitrary, $G_m$ is almost perfect.
	
	Now let $c$ be any constant.
	Choose an integer $m>c$.
	Then the almost
	perfect graph $G_m$ satisfies
	\[
	\chi(G_m)=\omega(G_m)+m>\omega(G_m)+c.
	\]
	Therefore no fixed constant $c$ can bound the additive gap
	$\chi(G)-\omega(G)$ over all almost perfect graphs.
	This completes the proof of Theorem~\ref{T2}.
\end{proof}

\vspace{6mm}

\n{\bf Acknowledgements:} 
We thank Dr. Hongzhang Chen for his helpful discussions and for bringing Refs. \cite{Ajtai1980,Shearer1983} to our attention.
The research is partially supported by the Foundation for Cultivated Young Talents of Fujian Province, China (Grant No. 2026350294), by the Natural Science Foundation of Fujian Province, China (Grant No. 2026J001968), by the Youth Foundation of Fujian Province (Grant No. JZ240035), by the Minnan Normal University Foundation (Grant No. KJ2023002), and by the Fujian Key Laboratory of Granular Computing and Applications (Minnan Normal University), the Institute of Meteorological Big Data-Digital Fujian, and the Fujian Key Laboratory of Data Science and Statistics.

\section*{Declarations}

%

\subsection*{Data availability}
Data sharing is not applicable to this article as no datasets were generated or analysed during the current study.

\subsection*{Conflict of interest}
The authors declare that they have no known competing financial interests or personal relationships that could have appeared to influence the work reported in this paper.


\begin{thebibliography}{99}


\bibitem{Ajtai1980}
M. Ajtai, J. Koml\'os, and E. Szemer\'edi,
A note on Ramsey numbers,
\textit{J. Combin. Theory Ser. A} \textbf{29} (1980), 354--360.

\bibitem{Berge1958}
C. Berge,
Sur le couplage maximum d'un graphe,
\textit{C. R. Acad. Sci. Paris} \textbf{247} (1958), 258--259.

\bibitem{BrianskiDaviesWalczak2024}
M. Bria\'nski, J. Davies, and B. Walczak,
Separating polynomial \(\chi\)-boundedness from \(\chi\)-boundedness,
\textit{Combinatorica} \textbf{44} (2024), 1--8.

\bibitem{chudnovsky2006}
M. Chudnovsky, N. Robertson, P. Seymour, and R. Thomas,
The strong perfect graph theorem,
\textit{Ann. of Math. (2)} \textbf{164} (2006), 51--229.

\bibitem{Erdos1959}
P. Erd\H{o}s,
Graph theory and probability,
\textit{Canad. J. Math.} \textbf{11} (1959), 34--38.

\bibitem{Esperet2017}
L. Esperet,
Graph colorings, flows and perfect matchings,
Habilitation Thesis, Universit\'e Grenoble Alpes, 2017.

\bibitem{GaspersHuang2019}
S. Gaspers and S. Huang,
$(2P_2,K_4)$-free graphs are 4-colorable,
\textit{SIAM J. Discrete Math.} \textbf{33} (2019), 1095--1120.

\bibitem{Gyarfas1975}
A. Gy\'arfas,
On Ramsey covering-numbers,
in: \textit{Infinite and Finite Sets (Colloq., Keszthely, 1973; Dedicated to P. Erd\H{o}s on His 60th Birthday)}, vol. II, Colloq. Math. Soc. J\'anos Bolyai \textbf{10} (1975), 801--816.


\bibitem{Gyarfas2023}
A. Gy\'arf\'as,
Problems close to my heart. With a preface by G. S\'ark\H{o}zy,
\textit{European J. Combin.} \textbf{111} (2023), 103695.



\bibitem{lovasz1972characterization}
L. Lov\'asz,
A characterization of perfect graphs,
\textit{J. Combin. Theory Ser. B} \textbf{13} (1972), 95--98.

\bibitem{Mycielski1955}
J. Mycielski,
Sur le coloriage des graphs,
\textit{Colloq. Math.} \textbf{3} (1955), 161--162.

\bibitem{scott2020induced}
A. Scott and P. Seymour,
Induced subgraphs of graphs with large chromatic number. VII. Gy\'arf\'as's complementation conjecture,
\textit{J. Combin. Theory Ser. B} \textbf{142} (2020), 43--55.

\bibitem{Shearer1983}
J. B. Shearer,
A note on the independence number of triangle-free graphs,
\textit{Discrete Math.} \textbf{46} (1983), 83--87.

\bibitem{Sivaraman2020}
V. Sivaraman,
25 interesting problems on $\chi$,
slide presentation, New perspectives in coloring and structure,
Mar. 15--20, 2020.

\bibitem{Spencer1977}
J. Spencer,
Asymptotic lower bounds for Ramsey functions,
\textit{Discrete Math.} \textbf{20} (1977), 69--76.

\bibitem{Tutte1947}
W. T. Tutte,
The factorization of linear graphs,
\textit{J. London Math. Soc.} \textbf{22} (1947), 107--111.

\bibitem{Wagon1980}
S. Wagon,
A bound on the chromatic number of graphs without certain induced subgraphs,
\textit{J. Combin. Theory Ser. B} \textbf{29} (1980), 345--346.







	











\bibitem{Zykov1949}
A. A. Zykov,
On some properties of linear complexes,
\textit{Mat. Sb. N.S.} \textbf{24(66)} (1949), 163--188.
      
\end{thebibliography}
\end{document}